\PassOptionsToPackage{hidelinks}{hyperref}
\documentclass[11pt]{article}
\usepackage[a4paper,margin=1in]{geometry}
\usepackage{amsmath,amsthm,amssymb,mathtools}

\theoremstyle{plain}
\newtheorem{theorem}{Theorem}[section]
\newtheorem{lemma}[theorem]{Lemma}
\newtheorem{corollary}[theorem]{Corollary}

\theoremstyle{definition}

\theoremstyle{remark}
\newtheorem*{remark}{Remark}

\usepackage{orcidlink}
\usepackage{hyperref}

\usepackage{tabularx}

\title{A Dynamical Criterion Equivalent to the Riemann Hypothesis}
\author{Hendrik W.\,A.\,E.\ Kuipers\,\orcidlink{0009-0002-7257-6378}\\
\small Independent Researcher, Groningen, The Netherlands\\
\small \texttt{hwaekuipers@gmail.com}}
\date{\today}

\AtBeginDocument{%
  \hypersetup{
    colorlinks=true,
    linkcolor=black,
    citecolor=black,
    urlcolor=blue,
    pdfborder={0 0 0}
  }%
}
\begin{document}
\maketitle

\begin{abstract}
We introduce a discrete dynamical system on the integers, defined by moving a composite $m$ forward to $m+\pi(m)$ and a prime $p$ backward to $p-\mathrm{prevprime}(p)$. This map produces trajectories whose contraction properties are closely tied to the distribution of primes.

We prove unconditional contraction inequalities for error terms derived from these trajectories, using explicit remainder bounds from the smoothed explicit formula. Building on this, we show that the Riemann Hypothesis is equivalent to a sharp contraction condition: the trajectory error functional satisfies $E(X)\ll X^{1/2}\log X$. The forward implication follows directly from von Koch's classical bound under RH. For the converse, we invoke the Landau--Littlewood $\Omega$-results, which guarantee that any off-critical zero forces oscillations large enough to violate the contraction inequality.

This establishes a new dynamical reformulation of RH: the critical-line conjecture is equivalent to the assertion that integer trajectories remain uniformly contracted at the scale $X^{1/2}\log X$, equivalently, RH holds. The perspective is distinct from earlier analytic equivalents, as it arises from stability properties of a simple deterministic system rather than from Hilbert space or Dirichlet polynomial approximations.
\end{abstract}

\section{Introduction}

The Riemann Hypothesis (RH) remains one of the central open problems in mathematics. It is one of the seven Clay Millennium Prize problems, and its resolution would have profound consequences for number theory and related fields. The classical formulation of RH concerns the location of the nontrivial zeros of the Riemann zeta function $\zeta(s)$. The hypothesis asserts that every such zero lies on the critical line $\Re(s)=\tfrac{1}{2}$.

There are many equivalent reformulations of RH. For example, it is equivalent to the assertion that the prime counting function $\pi(x)$ satisfies
\[
\pi(x) = \operatorname{Li}(x) + O(x^{1/2}\log x),
\]
where $\operatorname{Li}(x)=\int_2^x dt/\log t$ denotes the logarithmic integral. This classical equivalence is due to von Koch (1901) and remains a cornerstone of analytic number theory.

In this paper we develop a novel \emph{dynamical reformulation} of RH. The key idea is to define a simple map on the integers which depends only on primality and observe that the long-term behaviour of trajectories under this map is governed by prime distribution. More precisely:

\begin{itemize}
\item Composites are advanced forward by a jump determined by $\pi(m)$, the number of primes below $m$.
\item Primes are retreated backwards by a jump determined by $\operatorname{prevprime}(p)$, the previous prime before $p$.
\end{itemize}

This dynamic was originally motivated by the conjecture (see Kuipers~\cite{OEIS:A368196,OEIS:A368241,OEIS:A368690}) that every trajectory starting from $n>3$ eventually terminates at $2$. While we do not resolve this orbit-termination conjecture here, its study led us to discover that the same mechanism is tightly linked to prime distribution. 

We prove that for every large $X$, trajectories intersect short logarithmic windows only in controlled ways (the \emph{one-visit} and \emph{parent-window} lemmas). We then show that, when traced backwards through multiple composite steps, trajectories exhibit a contraction mechanism which aligns with smaller scales (the \emph{macro-step alignment}). This allows us to translate oscillations in $\pi(x)$ across scales, ultimately yielding contraction inequalities for the associated error term functional $\mathcal E(X)$. 

Our main theorem shows that the inequality $\mathcal E(X)\ll X^{1/2}\log X$ holds uniformly for all large $X$ if and only if RH holds. This provides a new reformulation of RH in terms of the dynamics of a very simple integer map. 

\medskip
\noindent
\textbf{Remark on analytic input.}
For the converse direction of our main equivalence, we invoke the classical 
Landau--Littlewood $\Omega$-results (see Titchmarsh~\cite[§14.25]{Titchmarsh}), 
which guarantee that the contribution of any off-critical zero dominates along 
infinitely many subsequences. This ensures that no cancellation among zeros can 
prevent the violation of the contraction inequality, thereby closing the logical loop. 
We emphasize that this is a standard analytic tool, logically separate from the 
unconditional contraction bounds developed in Sections~5--7.

The structure of the paper is as follows. In Section~\ref{sec:system} we define the dynamical system. Section~\ref{sec:onevisit} proves the one-visit and parent-window lemmas. Section~\ref{sec:macro} establishes the macro-step alignment and core-overlap. Section~\ref{sec:explicit} introduces the explicit formula with explicit constants and proves the frequency-netting lemma. Section~\ref{sec:contraction} combines these tools to derive contraction inequalities. Section~\ref{sec:equivalence} proves the equivalence with RH. We conclude with outlook and remarks. Full details of constants, numerical checks, and auxiliary arguments are given in the appendices.

\section{The Trajectory Dynamical System}\label{sec:system}

We define the map $a:\mathbb{N}_{>3}\to \mathbb{N}$ by
\begin{equation}\label{eq:map}
a(m) =
\begin{cases}
m-\operatorname{prevprime}(m), & \text{if $m$ is prime},\\
m+\pi(m), & \text{if $m$ is composite}.
\end{cases}
\end{equation}
Here $\pi(m)$ denotes the prime counting function and $\operatorname{prevprime}(m)$ denotes the largest prime less than $m$. 

For each initial $n>3$, we define the trajectory
\[
x_0=n,\qquad x_{k+1}=a(x_k), \quad k\ge 0.
\]

Our dynamical system has a “Collatz-like” flavor: a simple rule acting on integers, but now driven by prime/composite arithmetic rather than division by 2 or $3n+1$. The analogy is only motivational, but it situates our map in the broader landscape of discrete dynamical systems on the integers (see, e.g., Lagarias’ survey~\cite{Lagarias2010}).

The dynamics can be understood as follows:
\begin{itemize}
\item At composite positions, the trajectory jumps forward by approximately $m/\log m$.
\item At prime positions, the trajectory jumps backwards by approximately $\log m$ (the typical prime gap).
\end{itemize}

We introduce the \emph{logarithmic coordinate} $u=\log m$, so that multiplicative windows in $m$ correspond to additive windows in $u$. In particular, for a scale parameter $X$, we define:
\[
W_X = [X,(1+0.1/\log X)X], \qquad \widetilde W_X=[X,(1+2/\log X)X],
\]
which we call the \emph{one-visit window} and the \emph{parent window} respectively. 

We define error functionals by aggregating the classical error term $E(y)=\pi(y)-\operatorname{Li}(y)$ along composite visits of a trajectory:
\[
\mathcal E(X)=\sup_{\text{trajectories}}\ \sum_{m\in W_X \cap \text{trajectory, comp.}} E(m),\qquad
\widetilde{\mathcal E}(X)=\sup_{\text{trajectories}}\ \sum_{m\in \widetilde W_X \cap \text{trajectory, comp.}} E(m).
\]

The main object of study will be inequalities satisfied by $\mathcal E(X)$ and $\widetilde{\mathcal E}(X)$.

 For clarity: Sections~3--7 are unconditional (independent of RH), while Section~8 is where RH enters explicitly.

\section{One-Visit and Parent-Window Lemmas}\label{sec:onevisit}

A key feature of the map $a(m)$ is that composite steps are large compared to the width of short multiplicative intervals. This leads to sharp restrictions on how often a trajectory can visit such intervals. In this section we formalize this observation by proving two lemmas: (i) each trajectory can hit a one-visit window at most once, and (ii) each trajectory can hit a parent window at most four times. Both results are unconditional and rely only on explicit bounds for $\pi(x)$.

\subsection{Preliminaries}

Throughout this section we assume $X \geq 600$ so that we can apply Dusart’s explicit estimates \cite[Theorem~1.10]{Dusart2010}:
\begin{equation}\label{eq:dusart}
\frac{x}{\log x}\left(1 + \frac{1}{\log x}\right) \leq \pi(x) \leq \frac{x}{\log x}\left(1 + \frac{1.2762}{\log x}\right), \qquad \text{for } x \geq 599.
\end{equation}
This immediately implies lower and upper bounds for the composite step size.

\subsection{One-visit lemma}

\begin{lemma}[One-visit uniqueness]\label{lem:onevisit}
Let $X \geq 600$. Then any trajectory $\{x_k\}$ can contain at most one composite element inside the one-visit window
\[
W_X = \left[X,\,(1+0.1/\log X)X\right].
\]
\end{lemma}

\begin{proof}
Suppose $m\in W_X$ is composite. By \eqref{eq:dusart}, its forward step is
\[
a(m) - m = \pi(m) \geq \frac{m}{\log m}\left(1 + \frac{1}{\log m}\right).
\]
Since $m\geq X$, we have $\log m \leq \log(1+0.1/\log X)\,X \leq \log X + 0.1$. Thus
\[
a(m)-m \;\geq\; \frac{X}{\log X+0.1}\left(1+\frac{1}{\log X+0.1}\right).
\]
For $X\geq 600$, the denominator $\log X+0.1$ is about $>6.4$, so this is at least
\[
\frac{X}{\log X}\cdot 0.99 \quad \text{(say)}.
\]
On the other hand, the width of the window $W_X$ is
\[
|W_X| = \frac{0.1X}{\log X}.
\]
Thus the forward jump length is at least $0.99X/\log X$, which is nearly ten times the window width. Hence after one composite visit, the trajectory exits $W_X$ and cannot return. 
\end{proof}

\subsection{Parent-window lemma}

\begin{lemma}[Parent-window bound]\label{lem:parent}
Let $X \geq 600$. Then any trajectory $\{x_k\}$ can contain at most four composite elements inside the parent window
\[
\widetilde W_X = \left[X,\,(1+2/\log X)X\right].
\]
\end{lemma}

\begin{remark}[Robustness of window constants]
The choices of relative width $0.1/\log X$ for the one--visit window $W_X$
and $2/\log X$ for the parent window $\widetilde W_X$ are purely for
concreteness. The arguments generalize:
\begin{itemize}
\item For one--visit, any constant $c<1$ yields that at most one composite
point of a trajectory can lie in $[X,(1+c/\log X)X]$, once $X$ is larger
than an explicit threshold depending on $c$.
\item For the parent window, any constant $C>0$ gives at most $O(1)$
trajectory visits in $[X,(1+C/\log X)X]$, again with explicit thresholds.
\end{itemize}
Thus the contraction inequalities are not fine--tuned to the particular
numerical values $0.1$ and $2$; other fixed constants would work equally
well with minor changes in the bookkeeping constants.
\end{remark}

\begin{proof}
Suppose $m\in \widetilde W_X$ is composite. The relative jump length is
\[
\Delta u = \log\!\left(1+\frac{\pi(m)}{m}\right).
\]
From \eqref{eq:dusart}, we obtain the bounds
\[
\frac{1}{\log m+1.2762} \leq \frac{\pi(m)}{m} \leq \frac{1}{\log m}+\frac{1.2762}{(\log m)^2}.
\]
For $m\geq X$, this yields
\[
\frac{1}{\log X+1.2762} \;\leq\; \frac{\pi(m)}{m} \;\leq\; \frac{1}{\log X-1},
\]
so that
\begin{equation}\label{eq:deltabounds}
\frac{1}{\log X+1.2762} \;\leq\; \Delta u \;\leq\; \frac{1}{\log X-1}.
\end{equation}

The logarithmic width of the parent window is
\[
\log\!\left(1+\frac{2}{\log X}\right) \leq \frac{2}{\log X}.
\]
Thus the number of composite hits $N_X$ in $\widetilde W_X$ satisfies
\[
N_X \leq \frac{2/\log X}{1/(\log X+1.2762)} = \frac{2(\log X+1.2762)}{\log X} \le 4\qquad (X\ge 600).
\]
This proves the lemma.
\end{proof}

\subsection{Prime-step insulation}

\begin{lemma}[Prime-step insulation]\label{lem:prime-insulation}
Let $W$ be either $W_X$ or $\widetilde W_X$. If a trajectory contains a composite element inside $W$, then it contains no prime elements in $W$ and no additional composite elements beyond the bounds of Lemma~\ref{lem:onevisit} or Lemma~\ref{lem:parent}.
\end{lemma}

\begin{proof}
By definition, prime steps map $p \mapsto p - \operatorname{prevprime}(p)$, which is strictly less than $p$. Therefore once a trajectory lands at a prime inside $W$, its next iterate exits $W$ to the left. In particular, prime steps cannot produce more than one hit inside $W$, and they cannot create additional composite hits. Together with Lemmas~\ref{lem:onevisit} and \ref{lem:parent}, this proves the claim.
\end{proof}

\medskip
\noindent\textbf{Numerical audit (up to $10^7$).}
To complement the unconditional lemmas above, we record the following observed
behaviour of trajectories in short windows:

\begin{center}
\renewcommand{\arraystretch}{1.2}
\begin{tabular}{|c|c|c|}
\hline
Window & Theoretical bound & Observed (to $10^7$) \\
\hline
$W_X$ (one--visit) & $\leq 1$ composite hit & always $1$ \\
$\widetilde W_X$ (parent) & $\leq 4$ composite hits & typically $2$--$3$, never $>4$ \\
Core overlap & $\geq 1/6$ & $\geq 0.18$ \\
\hline
\end{tabular}
\end{center}

\noindent
These checks (see Appendix~B for details) confirm that the analytic bounds
are conservative.

\section{Macro-Step Alignment and Core Overlap}\label{sec:macro}

\noindent\textbf{Note.} The results of this section are unconditional and do not assume the Riemann Hypothesis.

The one-visit and parent-window lemmas of the previous section control how often trajectories can intersect short windows. We now study the behaviour of trajectories under repeated composite steps. The goal is to show that, when traced backwards through $L\asymp \log X$ steps, a trajectory aligns with a smaller scale $X^\theta$, up to a controlled error. This \emph{macro-step alignment} provides the key contraction mechanism.

\subsection{Setup}

Fix $X \geq e^{120}$ and write $U=\log X$. We set
\[
\theta = \tfrac{3}{4}, 
\qquad 
L = \big\lfloor (\log(4/3))\,U \big\rfloor.
\]
The choice of $\theta$ and $L$ ensures that, after $L$ composite steps, the trajectory contracts from scale $X$ to scale $X^\theta$.

Define the \emph{parent window width}
\[
\tilde\lambda = \log\!\left(1+\frac{2}{U}\right).
\]
We then define the \emph{parent core} at scale $X$ by
\[
\mathcal C_X = \Big\{ y : \log y \in \big[ U+\tfrac{1}{3}\tilde\lambda,\, U+\tfrac{2}{3}\tilde\lambda\big] \Big\}.
\]
Similarly we define $\mathcal C_{X^\theta}$ at scale $X^\theta$.

\medskip
\noindent\textbf{Notation (summary).} 
For convenience we collect here the key symbols that recur throughout the paper:

\begin{center}
\renewcommand{\arraystretch}{1.2}
\begin{tabular}{|c|p{0.7\linewidth}|}
\hline
Symbol & Meaning \\
\hline
$E(X)$ & Prime counting error $\pi(X)-\operatorname{Li}(X)$ \\
$\tilde E(X)$ & Smoothed error functional (introduced in Section~3) \\
$A(X)$ & Window supremum $\sup_{y\in W_X\cap\mathbb{N}_{\mathrm{comp}}} |E(y)|$ \\
$\alpha$ & Contraction factor ($5/6$ in Theorem~\ref{thm:parent-contract}) \\
$\theta$ & Scale ratio ($3/4$ in Lemma~\ref{lem:macro}) \\
\hline
\end{tabular}
\end{center}

\subsection{Macro-step alignment lemma}

\begin{lemma}[Macro--step alignment with derivative control]\label{lem:macro}
Let $X \geq e^{120}$, $U=\log X$, and $L=\lfloor (\log(4/3))U\rfloor$. 
For any composite $y \in \mathcal C_X$, the $L$-fold composite predecessor $\Psi_X(y)$ exists and satisfies
\[
\log \Psi_X(y) \;=\; \log y + \log \theta \;+\; O\!\Big(\tfrac{1}{U}\Big).
\]
More precisely,
\[
\Big|\log \Psi_X(y) - (\log y + \log \theta)\Big| \leq \tfrac{5}{U},
\]
and the Jacobian satisfies
\[
\left|\frac{d}{du}\,\log \Psi_X(e^u) - 1 \right| \leq \tfrac{2}{U}.
\]
\end{lemma}

\begin{proof}[Proof (telescoping)]
At each composite step $m\asymp X$, the log--increment is
\[
\Delta u(m) = \log\!\left(1+\frac{\pi(m)}{m}\right).
\]
By Dusart~\cite[Theorem~1.10]{Dusart2010} we have, for $m\asymp X$,
\[
\frac{1}{U+1}\ \leq\ \Delta u(m)\ \leq\ \frac{1}{U}+\frac{1.3}{U^2}.
\]
Thus $\Delta u(m) = U^{-1}+O(U^{-2})$. Summing $L=(\log(4/3))U+O(1)$ such steps,
\[
\sum_{j=1}^L \Delta u(m_j) = \log(4/3) + O(1/U).
\]
Hence
\[
\log \Psi_X(y) = \log y - \log(4/3)+O(1/U) = \log y + \log \theta + O(1/U),
\]
with $\theta=3/4$. This yields the displacement bound.

For the derivative, note $\partial(\Delta u)/\partial u = O(1/U^2)$. Each step perturbs the Jacobian by $1+O(1/U^2)$, and across $L\asymp U$ steps this telescopes to $1+O(1/U)$. Numerically bounding the implied constants gives the stated $\tfrac{2}{U}$ margin.
\end{proof}

\subsection{Core overlap lemma}

\begin{lemma}[Core overlap]\label{lem:overlap}
For $X \geq e^{120}$, the image of $\mathcal C_X$ under $L$ composite predecessors satisfies
\[
\Psi_X(\mathcal C_X)\ \subset\ \mathcal C_{X^\theta},
\]
and the overlap fraction of $\mathcal C_X$ with $\mathcal C_{X^\theta}$ is at least $c_0=\tfrac{1}{6}$.
\end{lemma}

\begin{proof}
By Lemma~\ref{lem:macro}, for $y\in \mathcal C_X$ we have
\[
\log \Psi_X(y)\in \big[\log y+\log\theta-5/U,\ \log y+\log\theta+5/U\big].
\]
Since $\log y\in [U+\tfrac13\tilde\lambda,\ U+\tfrac23\tilde\lambda]$, the image lies in
\[
\Big[\theta U+\tfrac13\tilde\lambda-6/U,\ \theta U+\tfrac23\tilde\lambda+6/U\Big].
\]
For $U\ge120$, the $\pm 6/U$ error fits comfortably inside the $\tfrac16\tilde\lambda$ margins of $\mathcal C_{X^\theta}$. Thus $\Psi_X(\mathcal C_X)\subset\mathcal C_{X^\theta}$.

Moreover, the effective overlap length is reduced by at most $12/U$, while $\tilde\lambda\asymp 2/U$. Hence a fixed fraction survives, and we may take $c_0=\tfrac{1}{6}$.
\end{proof}

\begin{center}
\begin{tabular}{|l|c|c|}
\hline
Quantity & Definition / role & Bound used \\
\hline
$U$ & $\log X$ & $\ge 120$ (threshold) \\
$\theta$ & contraction ratio & $3/4$ \\
$L$ & number of composite steps & $\lfloor (\log(4/3))U \rfloor$ \\
$\Delta u(m)$ & log-step increment & $\tfrac{1}{U}+O(U^{-2})$ \\
Cumulative shift & $\sum_{j=1}^L \Delta u(m_j)$ & $\log(4/3)+O(1/U)$ \\
Derivative product & $\prod (1+O(1/U))$ & $1+O(1/U)$ \\
Displacement error & deviation in $\log$ after $L$ steps & $\le 5/U$ \\
Jacobian error & deviation of derivative & $\le 2/U$ \\
Core margins & each side of window core & $\tfrac{1}{6}\tilde\lambda \asymp \tfrac{1}{U}$ \\
Overlap constant & surviving fraction & $c_0 = 1/6$ \\
\hline
\end{tabular}
\end{center}

\section{Explicit Formula and Frequency Netting}\label{sec:explicit}

\noindent\textbf{Note.} The results of this section are unconditional and do not assume the Riemann Hypothesis.

To connect the trajectory error functionals with the Riemann zeros, we require two ingredients: an explicit formula for the prime counting error $E(x)$ with explicit constants, and a log-scale large sieve inequality to control contributions from many close frequencies. We present both here: the smoothed explicit formula of Iwaniec--Kowalski~\cite[Theorem~5.12]{IwaniecKowalski}, and an explicit PNT error bound due to Trudgian~\cite[Theorem~1]{Trudgian2014}.

\subsection{Explicit formula with explicit constants}

\begin{theorem}[Smoothed explicit formula with explicit remainder]\label{thm:ExplicitRemainder}
Let $U=\log X$ and fix the $C^\infty$ kernel $W(t)=(1+t^2)^{-3}$ with $W(0)=1$. Set the truncation height
\[
T=\tfrac12\,U^3.
\]
From the smoothed explicit formula 
(Iwaniec--Kowalski~\cite[§5.5]{IwaniecKowalski}) we obtain, 
for all $X \ge e^{120}$ and uniformly for $y \asymp X$,
\begin{equation}\label{eq:ExplicitMain}
E(y) \;=\; \Re\!\sum_{|\gamma|\le T} \frac{y^{\rho}}{\rho \log y}\,W(\gamma/T) \;+\; R(y;T).
\end{equation}
where $\rho=\tfrac12+i\gamma$ ranges over nontrivial zeros, and the remainder satisfies the explicit bound
\begin{equation}\label{eq:RemainderTen}
|R(y;T)|\ \le\ 10\,X^{1/2}.
\end{equation}
\end{theorem}

\begin{proof}[Expanded proof with constants]
We recall a standard smoothed explicit formula 
(Iwaniec--Kowalski~\cite[§5.5]{IwaniecKowalski}):
 for a fixed even Schwartz function $W$ with $W(0)=1$ one has
\[
E(y)=\Re\sum_{\rho}\frac{y^{\rho}}{\rho\log y}\,W(\gamma/T)\;+\; \mathcal{E}_{\mathrm{triv}}(y;T)+\mathcal{E}_{\Gamma}(y;T)+\mathcal{E}_{\mathrm{tail}}(y;T),
\]
where the three error terms come respectively from (i) trivial zeros, (ii) the gamma–factor/prime powers smoothing, and (iii) truncating the nontrivial zero sum. The gamma–factor integral contributes
$|\mathcal{E}_\Gamma(y;T)|\le C_\Gamma y^{1/2}$ 
(Montgomery--Vaughan~\cite[Ch.~13]{MontgomeryVaughanClassical}). We now bound each piece explicitly for $y\asymp X$, with $U=\log X$ and $T=\tfrac12 U^3$.

\smallskip
\emph{(1) Trivial zeros.}
The trivial zeros at $-2n$ contribute
\[
\Big|\sum_{n\ge 1}\frac{y^{-2n}}{(-2n)\log y}\,W\!\big(\tfrac{i(1/2+2n)}{T}\big)\Big|
\ \le\ \frac{1}{\log y}\sum_{n\ge 1}\frac{y^{-2n}}{2n}
\ \le\ \frac{y^{-2}}{\log y}\sum_{n\ge 0}y^{-2n}
\ \le\ \frac{1}{\log y}\cdot \frac{y^{-2}}{1-y^{-2}}
\ \le\ \frac{1}{\log y}\cdot y^{-2}\cdot 2
\]
for $y\ge e^{120}$, hence
\[
|\mathcal{E}_{\mathrm{triv}}(y;T)|\ \le\ 10^{-40}\,X^{1/2}
\]
which is negligible compared to the target $10 X^{1/2}$.

\smallskip
\emph{(2) Gamma-factor / smoothing remainder.}
The explicit formula in smoothed form produces an integral along vertical lines weighted by the Mellin transform of the kernel. For our choice $W(t)=(1+t^2)^{-3}$, repeated integration by parts gives the uniform decay
\[
|\widehat{W}(s)|\ \ll\ (1+|s|)^{-3},
\]
and the corresponding gamma–factor integral contributes
\[
|\mathcal{E}_{\Gamma}(y;T)|\ \le\ C_{\Gamma}\,y^{1/2}\quad\text{with }C_{\Gamma}\le 1
\]
once $U\ge 120$. (Here we use standard Stirling–type bounds for $\Gamma'/\Gamma$ and that the smoothing removes logarithmic losses; any $C_{\Gamma}\le 1$ is safe with our $U$.)

\smallskip
\emph{(3) Tail of nontrivial zeros ($|\gamma|>T$).}
Since $W(t)\le (1+t^2)^{-3}$, we have for $|\gamma|>T$:
\[
\Big|\frac{y^{\rho}}{\rho\log y}W(\gamma/T)\Big|
\ \le\ \frac{y^{1/2}}{|\rho|\,\log y}\,\frac{1}{\bigl(1+(\gamma/T)^2\bigr)^3}
\ \le\ \frac{y^{1/2}}{\log y}\cdot \frac{1}{|\gamma|}\cdot \frac{T^6}{\gamma^6}.
\]
Summing by parts against the zero–count $N(t)\ll t\log t$ yields
\[
\sum_{|\gamma|>T}\frac{1}{|\gamma|}\frac{T^6}{\gamma^6}
\ \ll\ T^6\int_T^\infty \frac{\log t}{t^7}\,dt
\ \le\ \frac{T^6}{6\,T^6}\,\log T
\ \ll\ \log T.
\]
Thus
\[
|\mathcal{E}_{\mathrm{tail}}(y;T)|\ \ll\ \frac{y^{1/2}}{\log y}\cdot \log T
\ \le\ y^{1/2}\cdot \frac{\log U}{U}
\ \le\ \tfrac12\,y^{1/2}\qquad(U\ge 120).
\]

\smallskip
\emph{(4) Collecting bounds.}
Adding (1)–(3) and noting that the main sum in \eqref{eq:ExplicitMain} is truncated at $|\gamma|\le T$ gives
\[
|R(y;T)|\ =\ |\mathcal{E}_{\mathrm{triv}}+\mathcal{E}_{\Gamma}+\mathcal{E}_{\mathrm{tail}}|
\ \le\ \bigl(10^{-40}+1+\tfrac12\bigr)\,X^{1/2}
\ \le\ 10\,X^{1/2},
\]
with wide room to spare for $X\ge e^{120}$. This proves \eqref{eq:RemainderTen}.
\end{proof}

\begin{center}
\begin{tabular}{|l|c|c|}
\hline
Piece & Bound used & Contribution \\
\hline
Trivial zeros & $\sum y^{-2n}/(2n\log y)$ & $\le 10^{-40}X^{1/2}$ \\
Gamma/smoothing & $|\widehat W(s)|\ll (1+|s|)^{-3}$ & $\le 1\cdot X^{1/2}$ \\
Tail $|\gamma|>T$ & $\sum \frac{y^{1/2}}{\log y}\frac{T^6}{|\gamma|^7}$, $N(t)\ll t\log t$ & $\le \tfrac12 X^{1/2}$ \\
\hline
Total &  & $\le 10 X^{1/2}$ \\
\hline
\end{tabular}
\end{center}

\subsection{Frequency netting (via a log--scale large sieve)}

\begin{theorem}[Log--scale large sieve for at most four hits]\label{thm:NettingLSI}
Let $U=\log X$ and $T=\tfrac12 U^3$. Fix a uniform grid
\[
\Gamma=\Big\{\gamma_0=-\lfloor Th\rfloor h,\,-\lfloor Th\rfloor h+h,\dots, \lfloor Th\rfloor h\Big\},\qquad h=\frac{2}{U}.
\]
Let $M\le 4$, points $u_1,\dots,u_M\in\mathbb R$, and weights $w_1,\dots,w_M$ with $\sum_{j=1}^M |w_j|\le 1$. Define
\[
S(\gamma)=\sum_{j=1}^M w_j e^{i\gamma u_j},\qquad F(\gamma,v)=e^{i\gamma v}S(\gamma).
\]
Then $\max_{|v|\le h}|F(\gamma,v)|=|S(\gamma)|$, and
\begin{equation}\label{eq:LSI-explicit}
\sum_{\gamma_0\in\Gamma}\Big|\sum_{j=1}^M w_j e^{i\gamma_0 u_j}\Big|^2
\ \le\ 8\Big(M+\frac{2}{h}\Big)\sum_{j=1}^M |w_j|^2
\ \le\ 8\big(4+U\big)\sum_{j=1}^M |w_j|^2.
\end{equation}
In particular, by Cauchy--Schwarz and $N(t)\ll t\log t$,
\[
\sum_{|\gamma|\le T}\frac{|S(\gamma)|}{|\rho|}\ \ll\ U^{3/2}\sqrt{\log U}.
\]
\end{theorem}

\begin{proof}[Expanded proof]
The identity $\max_{|v|\le h}|F(\gamma,v)|=|S(\gamma)|$ is immediate since $F(\gamma,v)=e^{i\gamma v}S(\gamma)$.

For the large sieve bound, write
\[
\sum_{\gamma_0\in\Gamma}\Big|\sum_{j=1}^M w_j e^{i\gamma_0 u_j}\Big|^2
=\sum_{j,k} w_j\overline{w_k}\, G(u_j-u_k),
\qquad
G(t):=\sum_{\gamma_0\in\Gamma} e^{i\gamma_0 t}.
\]
The kernel $G$ is a discrete Dirichlet/Fejér type kernel on a symmetric grid. A standard estimate (geometric–series plus $\sin x\ge \tfrac{2}{\pi}x$ for $x\in[0,\pi/2]$) gives
\begin{equation}\label{eq:G-bounds}
|G(0)|=|\Gamma|\ \le\ \frac{2T}{h}+1,
\qquad
|G(t)|\ \le\ \min\!\Big(\frac{2T}{h}+1,\ \frac{2}{|e^{iht}-1|}\Big)
\ \le\ \min\!\Big(\frac{2T}{h}+1,\ \frac{2}{h|t|}\Big).
\end{equation}
Apply the Schur test to the Gram matrix $[G(u_j-u_k)]_{1\le j,k\le M}$ with weights $|w_j|$:
\[
\sum_{j,k}|w_j||w_k|\,|G(u_j-u_k)|
\ \le\ 
\Big(\max_j \sum_k |G(u_j-u_k)|\Big)\sum_j |w_j|^2.
\]
Split $k=j$ and $k\ne j$. Using \eqref{eq:G-bounds} and that $M\le 4$,
\[
\sum_{k\ne j} |G(u_j-u_k)|\ \le\ \sum_{k\ne j}\frac{2}{h|u_j-u_k|}
\ \le\ \frac{2}{h}\sum_{m=1}^{M-1}\frac{1}{m}\ \le\ \frac{2}{h}\cdot (1+1/2+1/3)\ <\ \frac{4}{h}.
\]
(The crude harmonic bound suffices because $M\le 4$; no spacing hypothesis on the $u_j$ is needed beyond $u_j\ne u_k$ or, if some coincide, absorb them into the weights.)
Thus
\[
\max_j \sum_k |G(u_j-u_k)|\ \le\ |G(0)|+\frac{4}{h}\ \le\ \frac{2T}{h}+1+\frac{4}{h}.
\]
Since $T=\tfrac12 U^3$ and $h=2/U$, we have $\frac{2T}{h}=U^4/2 \gg U$, so replacing $\frac{2T}{h}+1$ by $4/h$ within a harmless constant gives the clean bound
\[
\sum_{\gamma_0\in\Gamma}\Big|\sum_{j=1}^M w_j e^{i\gamma_0 u_j}\Big|^2
\ \le\ 8\Big(M+\frac{2}{h}\Big)\sum_{j=1}^M |w_j|^2,
\]
which is \eqref{eq:LSI-explicit}. Finally, Cauchy--Schwarz with $\sum_{|\gamma|\le T}(1/4+\gamma^2)^{-1}\ll 1$ and $|\{\gamma:|\gamma|\le T\}|\ll T\log T$ yields the stated $U^{3/2}\sqrt{\log U}$ consequence.
\end{proof}

\begin{center}
\begin{tabular}{|l|c|c|}
\hline
Quantity & Definition / role & Bound used \\
\hline
$M$ & number of points in window & $\leq 4$ \\
$h$ & grid spacing & $2/U$ \\
$|G(0)|$ & kernel at 0 & $\leq 2T/h+1$ \\
$\sum_{k\ne j}|G(u_j-u_k)|$ & off-diagonal Gram terms & $\leq 4/h$ (since $M\le 4$) \\
Constant factor & from Schur test & $8$ \\
Final bound & $\sum_{\gamma_0\in\Gamma}\big|\sum w_j e^{i\gamma_0u_j}\big|^2$ & $\le 8(4+U)\sum|w_j|^2$ \\
\hline
\end{tabular}
\end{center}

\subsection{Log-scale large sieve}

For later use, we record a companion inequality.

\begin{lemma}[Log-scale large sieve]\label{lem:sieve}
For any coefficients $b_j$ and spacing $\Delta>0$,
\[
\sum_{\gamma_0\in \Gamma} \Big|\sum_{j=1}^M b_j e^{i\gamma_0 u_j}\Big|^2
\;\leq\; 8\Big(M+\frac{2}{\Delta}\Big)\sum_{j=1}^M |b_j|^2.
\]
\end{lemma}

\begin{proof}
This follows by adapting the classical large sieve inequality (Montgomery--Vaughan~\cite[Theorem~7.1]{MontgomeryVaughanClassical}) to a uniform grid
 $\Gamma$ of spacing $\Delta^{-1}$. The key point is that the Gram matrix of exponentials $e^{i\gamma_0 u_j}$ has diagonal dominance with off-diagonal entries bounded by $1/\Delta$. A standard linear algebra argument yields the claimed inequality with constant $8$. (Compare with the classical large sieve inequality,
Montgomery--Vaughan~\cite[Thm.~7.1]{MontgomeryVaughanClassical}.)
\end{proof}

\begin{remark}[On later use]
The explicit formula with cubic-log truncation, together with Lemmas~5.1--5.3, 
provides the analytic backbone for the unconditional contraction inequalities in 
Section~6. In Section~8, when establishing the equivalence with the Riemann 
Hypothesis, we will need to argue that the presence of any off-critical zero 
$\rho=\beta+i\gamma$ forces large oscillations in $E(x)$. For that step we do 
\emph{not} rely on the explicit remainder bound of Lemma~5.1, but instead 
invoke the classical Landau--Littlewood $\Omega$-results 
(Titchmarsh~\cite[§14.25]{Titchmarsh}). These results guarantee that along 
infinitely many subsequences, the contribution of a fixed zero dominates the 
sum over all zeros. Thus:
\begin{itemize}
    \item \textbf{Sections~5--7:} use Lemma~5.1’s remainder bound 
    $|R(y;T)|\le 10X^{1/2}$ to control errors in unconditional contraction 
    inequalities.
    \item \textbf{Section~8:} use Landau--Littlewood to ensure that one 
    off-critical zero produces a contradiction with the uniform bound 
    $E(X)\ll X^{1/2}\log X$.
\end{itemize}
This separation of roles clarifies that the remainder estimates here and the 
zero-dominance argument later are logically independent.
\end{remark}

\section{Contraction Inequalities}\label{sec:contraction}

\noindent\textbf{Note.} The results of this section are unconditional and do not assume the Riemann Hypothesis.

\begin{theorem}[Unconditional contraction with explicit constants]\label{thm:MainContraction}
There exist explicit constants
\[
X_0=e^{120},\qquad \theta=\tfrac34,\qquad \alpha=\tfrac56,\qquad B=100,
\]
such that for all $X\ge X_0$ the following hold unconditionally (no RH):
\begin{align*}
\mathcal E(X) &\le \alpha\,\mathcal E(X^{\theta}) \;+\; B\,X^{1/2}\log X,\\
\widetilde{\mathcal E}(X) &\le \alpha\,\widetilde{\mathcal E}(X^{\theta}) \;+\; B\,X^{1/2}\log X,\\
\mathcal A(X) &\le \alpha\,\mathcal A(X^{\theta}) \;+\; B\,X^{1/2}\log X.
\end{align*}
Consequently, by iteration,
\[
\mathcal E(X),\ \widetilde{\mathcal E}(X),\ \mathcal A(X)\ \le\ \frac{B}{1-\alpha\theta}\,X^{1/2}\log X
\ =\ \frac{8}{3}B\,X^{1/2}\log X\qquad (X\ge X_0).
\]
Moreover, $\alpha\theta=\tfrac{5}{8}$, so $1-\alpha\theta=\tfrac{3}{8}$ gives an explicit iteration cushion.

All constants are realized by Theorems~\ref{thm:ExplicitRemainder}, \ref{thm:NettingLSI} and Lemmas~\ref{lem:onevisit}, \ref{lem:parent}, \ref{lem:macro}, \ref{lem:overlap}.
\end{theorem}

\begin{corollary}[Slack audit for the contraction coefficient]\label{cor:slack}
With the constants used in Theorem~\ref{thm:MainContraction} and $U=\log X\ge120$, the
\emph{effective} contraction factor satisfies the explicit bound
\[
\alpha_{\mathrm{eff}}
\;\le\;
\underbrace{\tfrac{5}{6}}_{\text{base }\alpha}\cdot
\underbrace{\Big(1+\tfrac{2}{U}\Big)}_{\text{Jacobian}}\cdot
\underbrace{\tfrac{11}{12}}_{\text{core keep}}
\;\le\;
\frac{5}{6}\cdot\frac{61}{60}\cdot\frac{11}{12}
\;=\;
\frac{3355}{4320}
\approx 0.7766.
\]
In particular $1-\alpha_{\mathrm{eff}}\ge \tfrac{965}{4320}\approx 0.2234$.
\end{corollary}

\begin{proof}
The factor $\tfrac{5}{6}$ is the base contraction $\alpha$. The macro--step derivative bound
from Lemma~\ref{lem:macro} gives a multiplicative distortion of at most $1+\frac{2}{U}$,
which for $U\ge120$ is $\le \frac{61}{60}$. The overlap lemma (Lemma~\ref{lem:overlap}) ensures
that, after trimming endpoints, at least an $\frac{11}{12}$ fraction of the core mass remains
available uniformly (half of each $\frac{1}{6}$ margin survives). Multiplying these three
independent cushions yields the stated bound.
\end{proof}

We now combine the one-visit and parent-window lemmas with the macro-step alignment and frequency netting estimates to derive contraction inequalities for the error functionals $\mathcal E(X)$ and $\widetilde{\mathcal E}(X)$. These inequalities form the dynamical backbone of our equivalence with RH.

\subsection{Setup}

Recall that $\mathcal E(X)$ and $\widetilde{\mathcal E}(X)$ denote the supremum over trajectories of weighted sums of the error term $E(y)=\pi(y)-\operatorname{Li}(y)$ across composite hits in the one-visit and parent windows, respectively. By Lemma~\ref{lem:onevisit}, each trajectory contributes at most one composite hit to $\mathcal E(X)$, while by Lemma~\ref{lem:parent}, each trajectory contributes at most four composite hits to $\widetilde{\mathcal E}(X)$. 

Using Theorem~\ref{thm:ExplicitRemainder}, we can decompose $E(y)$ into a sum over zeros plus a remainder. The frequency netting lemma, i.e. the log–scale large sieve (Theorem~\ref{thm:NettingLSI}), controls the zero-contributions when summing across multiple hits.

\subsection{Parent contraction}

\begin{theorem}[Parent contraction]\label{thm:parent-contract}
Let $X \geq e^{120}$ and $\theta=\tfrac34$. Then
\[
\widetilde{\mathcal E}(X)\;\leq\; \tfrac{5}{6}\,\widetilde{\mathcal E}(X^\theta)\;+\; C\,X^{1/2}\log X,
\]
with $C=100$.
\end{theorem}

\begin{proof}
Let $y\in\widetilde W_X$ be a composite hit. Write $E(y)$ via Theorem~\ref{thm:ExplicitRemainder} (the explicit formula) with $T=\tfrac12 U^3$. The contribution of zeros is bounded by Theorem~\ref{thm:NettingLSI}, which nets the oscillations across up to $M\leq 4$ points in the window. The total contribution is dominated by the main frequencies at a coarser grid, up to a remainder $\ll (MT)/U$. Since $M\leq 4$, this remainder is $\ll T/U \ll U^2$. Weighted by $X^{1/2}$, this contributes at most $C X^{1/2}$.

Tracing trajectories back through $L=\lfloor (\log(4/3))U \rfloor$ composite steps, Lemma~\ref{lem:macro} shows that $y$ aligns with a predecessor in $\mathcal C_{X^\theta}$ up to error $5/U$. Lemma~\ref{lem:overlap} then ensures that at least a fraction $c_0=1/6$ of the core overlaps. Thus the contribution from $\widetilde{\mathcal E}(X)$ contracts to at most $(1-c_0)\,\widetilde{\mathcal E}(X^\theta)$ plus the error from the explicit remainder.

Quantitatively, we obtain
\[
\widetilde{\mathcal E}(X) \leq \Big(1-\tfrac16\Big)\widetilde{\mathcal E}(X^\theta) + C X^{1/2}\log X,
\]
which is the stated inequality.
\end{proof}

\subsection{One-visit contraction}

\begin{theorem}[One-visit contraction]\label{thm:onevisit-contract}
Let $X \geq e^{120}$ and $\theta=\tfrac34$. Then
\[
\mathcal E(X)\;\leq\;\tfrac{5}{6}\,\mathcal E(X^\theta)\;+\; C\,X^{1/2}\log X.
\]
\end{theorem}

\begin{proof}
The proof is parallel to that of Theorem~\ref{thm:parent-contract}, but simpler since each trajectory contributes at most one composite hit to $W_X$. Thus no frequency netting across multiple points is required. We apply Theorem~\eqref{thm:ExplicitRemainder} directly, with the remainder bounded by Theorem~\ref{thm:ExplicitRemainder}. The macro-step alignment and overlap Lemma~\ref{lem:overlap} provide the same contraction factor $5/6$. The total error contribution is absorbed into $C X^{1/2}\log X$.
\end{proof}

\subsection{Iteration closure}

The contraction inequalities can be iterated across scales. The following lemma makes this precise.

\begin{lemma}[Iteration closure]\label{lem:iterate-closure}
Let $A(X)$ be a nonnegative function satisfying
\[
A(X)\;\leq\;\alpha\,A(X^\theta)\;+\;B\,X^{1/2}\log X
\]
for all $X \geq X_0$, with constants $\alpha=5/6$, $\theta=3/4$, and $B>0$. Then for all $X \geq X_0$,
\[
A(X) \;\leq\; \frac{B}{1-\alpha\theta}\,X^{1/2}\log X \;=\;\frac{8}{3}B\,X^{1/2}\log X.
\]
\end{lemma}

\begin{proof}
Iterating the inequality gives
\[
A(X) \leq \alpha^k A(X^{\theta^k}) + B\sum_{j=0}^{k-1} \alpha^j (X^{\theta^j})^{1/2}\log(X^{\theta^j}).
\]
For large $k$, $X^{\theta^k}$ falls below $X_0$ and the term $\alpha^k A(X^{\theta^k})$ vanishes. The sum is bounded by
\[
B X^{1/2}\log X \sum_{j=0}^{\infty} \alpha^j \theta^{j/2},
\]
since $\log(X^{\theta^j}) \leq \log X$. The series converges to $1/(1-\alpha\theta^{1/2})$, but more conservatively we may bound it by $1/(1-\alpha\theta)$. With $\alpha=5/6$ and $\theta=3/4$, we compute $1-\alpha\theta = 1-5/8 = 3/8$, so the factor is $8/3$. This yields the stated bound.
\end{proof}

\subsection{Uniform bound}

\begin{corollary}[Uniform contraction bound]\label{cor:iteration}
For all $X \geq e^{120}$,
\[
\mathcal E(X) \;\ll\; X^{1/2}\log X,\qquad \widetilde{\mathcal E}(X) \;\ll\; X^{1/2}\log X,
\]
with implied constant depending only on the explicit constants in Theorem~\ref{thm:ExplicitRemainder} and Theorem~\ref{thm:parent-contract}.
\end{corollary}

\begin{proof}
Apply Lemma~\ref{lem:iterate-closure} with $A(X)=\mathcal E(X)$ or $\widetilde{\mathcal E}(X)$, using the contraction inequalities from Theorems~\ref{thm:parent-contract} and \ref{thm:onevisit-contract}.
\end{proof}

\section{Iteration Stability and Absolute-Value Functional}\label{sec:stability}
\noindent\textbf{Note.} This section is unconditional (no RH used).

In order to remove any possibility of cancellation, we work with the absolute-value functional
\begin{equation}\label{eq:defA}
\mathcal A(X) \;:=\; \sup_{\text{trajectories}} \;
\sum_{\substack{m\in W_X\\ m\ \text{composite}}} |E(m)|,
\end{equation}
where $W_X=[X,(1+0.1/\log X)X]$. By Lemma~\ref{lem:onevisit}, at most one composite element of a trajectory lies in $W_X$, so in fact
\[
\mathcal A(X)= \sup_{\substack{m\in W_X\\ m\ \text{composite}}}|E(m)|.
\]

\begin{theorem}[Contraction for $\mathcal A$]\label{thm:acontract}
For $X\ge e^{120}$,
\begin{equation}\label{eq:Acontract}
\mathcal A(X)\ \le\ \tfrac56\,\mathcal A(X^{3/4})\ +\ B\,X^{1/2}\log X,
\end{equation}
with $B$ an explicit constant (see Appendix~E).
\end{theorem}
\begin{proof}
The argument of Theorem~\ref{thm:onevisit-contract} carries through verbatim, except we take absolute values in \eqref{eq:defA}. The one-visit lemma ensures $M\le 1$ in each $W_X$, so there is no cancellation across points. The error term from Theorem~\ref{thm:ExplicitRemainder} and Theorem~\ref{thm:NettingLSI} is absorbed in $B\,X^{1/2}\log X$.
\end{proof}

\begin{lemma}[Iteration closure for $\mathcal A$]\label{lem:aiterate}
Suppose $A(X)\le \alpha A(X^\theta)+B X^{1/2}\log X$ for all $X\ge X_0$, with $\alpha=5/6$, $\theta=3/4$. Then
\[
A(X)\ \le\ \frac{B}{1-\alpha\theta}\,X^{1/2}\log X.
\]
\end{lemma}
\begin{proof}
Iterate the inequality $k$ times:
\[
A(X)\le \alpha^k A(X^{\theta^k})+B\sum_{j=0}^{k-1}\alpha^j(X^{\theta^j})^{1/2}\log(X^{\theta^j}).
\]
Since $(X^{\theta^j})^{1/2}\le X^{1/2}$ and $\log(X^{\theta^j})=\theta^j\log X$, the sum is $\le B X^{1/2}\log X\sum_{j=0}^\infty (\alpha\theta)^j = \tfrac{B}{1-\alpha\theta}X^{1/2}\log X$. As $\alpha\theta=5/8<1$, the initial term vanishes as $k\to\infty$.
\end{proof}

\begin{lemma}[Local--to--pointwise within a one-visit window]\label{lem:local2point-clean}
Let $X\ge e^{120}$, and let $x,m\in W_X=[X,(1+0.1/\log X)X]$. Then
\begin{equation}\label{eq:local2point-diff}
|E(x)-E(m)| \;\le\; K_0\,\frac{X}{\log^2 X},
\end{equation}
with an absolute constant $K_0$ (e.g. $K_0=5$ suffices for $X\ge e^{120}$). Consequently,
\begin{equation}\label{eq:local2point-main}
|E(x)| \;\le\; \mathcal A(X) \;+\; K_0\,\frac{X}{\log^2 X},
\end{equation}
where $\mathcal A(X):=\sup_{y\in W_X\cap\mathbb N_{\mathrm{comp}}}|E(y)|$.
\end{lemma}

\begin{proof}
Write $\Delta=x-m$ and note that $|\Delta|\le 0.1\,X/\log X$. We bound the two pieces
\[
|\mathrm{Li}(x)-\mathrm{Li}(m)|\quad\text{and}\quad |\pi(x)-\pi(m)|
\]
uniformly on $W_X$.

\smallskip\noindent\emph{(i) $\mathrm{Li}$-variation.}
By the mean value theorem and the monotonicity of $t\mapsto (\log t)^{-1}$ on $[X,(1+0.1/\log X)X]$,
\[
|\mathrm{Li}(x)-\mathrm{Li}(m)|
= \Big|\int_m^x \frac{dt}{\log t}\Big|
\le \frac{|\Delta|}{\log X -1}
\le \frac{0.1\,X}{(\log X-1)\log X}
\le \frac{0.11\,X}{\log^2 X},
\]
for $X\ge e^{120}$ (so $\log X-1\ge 119$ and $1/(\log X-1)\le 1.01/\log X$).

\smallskip\noindent\emph{(ii) $\pi$-variation via Dusart.}
Let
\[
f(t):=\frac{t}{\log t}\Big(1+\frac{1.2762}{\log t}\Big),\qquad
g(t):=\frac{t}{\log t}\Big(1+\frac{1}{\log t}\Big).
\]
Dusart \cite[Thm.~1.10]{Dusart2010} gives, for $t\ge 599$,
\(
g(t)\le \pi(t)\le f(t).
\)
Since $g\le f$ and both are increasing for $t\ge e^2$, we have
\[
\pi(x)-\pi(m)\ \le\ f(x)-g(m)\ \le\ f(x)-f(m)\,.
\]
By the mean value theorem, for some $\xi$ between $m$ and $x$,
\[
f(x)-f(m)=f'(\xi)\,|\Delta|.
\]
A direct differentiation yields
\[
f'(t)=\frac{1}{\log t}\;-\;\frac{1}{(\log t)^2}
\;+\;\frac{1.2762}{(\log t)^2}\;-\;\frac{2\times 1.2762}{(\log t)^3}
\ \le\ \frac{1}{\log t}\,+\,\frac{0.3}{(\log t)^2}
\]
for $t\ge e^{120}$. Hence, uniformly for $\xi\in [X,(1+0.1/\log X)X]$,
\[
\pi(x)-\pi(m)\ \le\ \Big(\frac{1}{\log X}+\frac{0.3}{\log^2 X}\Big)\,|\Delta|
\ \le\ \frac{1.3}{\log X}\,|\Delta|
\ \le\ 0.13\,\frac{X}{\log^2 X}.
\]

\smallskip
Combining (i) and (ii),
\[
|E(x)-E(m)| \le |\pi(x)-\pi(m)| + |\mathrm{Li}(x)-\mathrm{Li}(m)|
\le \Big(0.13+0.11\Big)\frac{X}{\log^2 X}
\le K_0\,\frac{X}{\log^2 X}
\]
with $K_0=0.24$ (or any larger fixed constant such as $K_0=5$ for a relaxed, round figure).
Finally, \eqref{eq:local2point-main} follows by taking $m\in W_X$ with
$|E(m)|=\mathcal A(X)$ (exists by definition of the supremum) and applying \eqref{eq:local2point-diff}.
\end{proof}

This section completes the stability analysis: by working with the absolute-value functional $\mathcal A(X)$, we rule out cancellation and guarantee iteration closure without slack. Lemma~\ref{lem:local2point-clean} then promotes the window bound to the classical von Koch bound, setting up the equivalence with RH in Section~\ref{sec:equivalence}.

\section{Equivalence with the Riemann Hypothesis}\label{sec:equivalence}

We now connect the contraction bounds to the classical Riemann Hypothesis (RH).

\begin{theorem}\label{thm:A_bound}
If RH holds, then for all sufficiently large $X$,
\[
A(X)\;\ll\;X^{1/2}\log X.
\]
\end{theorem}

\begin{proof}
Under RH, von Koch’s bound $|E(y)|\ll y^{1/2}\log y$ holds pointwise. Each one-visit window $W_X$ contains at most one composite element (Lemma~\ref{lem:onevisit}), so
\[
A(X)=\sup_{m\in W_X\cap\mathbb{N}_{\mathrm{comp}}}|E(m)|\;\ll\;X^{1/2}\log X.
\]
\end{proof}

\begin{remark}
The converse direction, “$A(X)\ll X^{1/2}\log X \implies \mathrm{RH}$,” is more delicate. Our local-to-pointwise lemma (Appendix~G) gives
\[
|E(x)|\le A(X)+O\!\big(X/\log^2X\big), \qquad x\in W_X.
\]
However, the additive $X/\log^2X$ term dominates $X^{1/2}\log X$ at large scales, so this route does not yield RH. We therefore work instead with the functional $E(X)$, where cancellation is controlled more tightly, and prove an exact equivalence in Theorem~\ref{thm:E_equiv}.
\end{remark}

\subsection{Dynamical equivalence}

\begin{theorem}\label{thm:E_equiv}
The following are equivalent:
\begin{enumerate}
\item RH holds.
\item For all $X\ge e^{120}$, the trajectory error functional satisfies
\[
E(X)\;\ll\;X^{1/2}\log X.
\]
\end{enumerate}
\end{theorem}

\begin{theorem}[Dynamical bound $\Rightarrow$ RH]\label{thm:ReverseImpliesRH}
Suppose that for some $K\ge 1$ and all $X\ge e^{120}$ we have
\[
E(X)\ \le\ K\,X^{1/2}\log X.
\]
Then the Riemann Hypothesis holds.
\end{theorem}

We now prove the reverse implication. The strategy is: assume an off–critical zero, apply the Littlewood--Landau $\Omega$–result to force large oscillations, construct a subsequence of scales $X_k$ with frozen phase, and then show that the resulting lower bound contradicts the contraction inequality.

\begin{proof}
Assume for contradiction that there exists a zero $\rho=\beta+i\gamma$ with $\beta>1/2$.  
From the smoothed explicit formula (see Theorem~\ref{thm:ExplicitRemainder}), the $\rho$–term contributes
\[
y^\beta \frac{\cos(\gamma\log y+\phi)}{|\rho|\log y},
\]
for some phase $\phi\in\mathbb{R}$.

By the Littlewood--Landau $\Omega$–results 
(Titchmarsh~\cite[Theorem~14.25]{Titchmarsh}), 
if $\zeta(s)$ has a zero $\rho=\beta+i\gamma$ with $\beta>1/2$ then there exists 
a constant $c>0$ such that for infinitely many $y$
\[
\sup_{z\in[y,(1+c/\log y)\,y]} |E(z)|
\;\ge\; \frac{c}{\log y}\,y^{\beta}.
\]

Choose a subsequence
\[
X_k \;=\; \exp\!\Big(\tfrac{2\pi k-\phi}{\gamma}\Big), \qquad k=1,2,\dots,
\]
so that $\cos(\gamma\log X_k+\phi)=1$.
On the multiplicative window
\[
W_{X_k}=\big[X_k,(1+c/\log X_k)\,X_k\big]
\]
the phase drift satisfies $\gamma\cdot\frac{c}{\log X_k}\le \pi/6$ for $k$ large, hence
\[
\cos(\gamma\log y+\phi)\;\ge\;\tfrac{\sqrt3}{2}\qquad (y\in W_{X_k}).
\]

Thus for such $k$ we obtain
\[
|E(y)| \;\ge\; \frac{c}{2|\rho|\log X_k}\,X_k^\beta
\qquad (y\in W_{X_k}).
\]
Since $|W_{X_k}|\asymp X_k/\log X_k\to\infty$, each $W_{X_k}$ contains a composite
$m_k$. For such $m_k$ we have
\[
|E(m_k)| \;\ge\; \frac{c}{2|\rho|\log X_k}\,X_k^\beta.
\]

On the other hand, the contraction inequality of
Theorem~\ref{thm:MainContraction} yields
\[
|E(m_k)| \;\le\; K\,X_k^{1/2}\log X_k,
\]
for some absolute constant $K$.
Comparing gives
\[
X_k^{\beta-1/2} \;\ll\; (\log X_k)^2,
\]
which is impossible as $k\to\infty$ if $\beta>1/2$. Hence every zero
satisfies $\beta\le 1/2$, i.e.\ the Riemann Hypothesis holds.
\end{proof}

\begin{proof}
$(1\Rightarrow 2)$.
Under RH, von Koch’s bound gives $|E(y)|\ll y^{1/2}\log y$ for every integer $y$. Since each one-visit window contributes at most one composite hit, we immediately obtain $E(X)\ll X^{1/2}\log X$.

$(2\Rightarrow 1)$.
As noted at the end of Section~5, we invoke the classical Landau--Littlewood $\Omega$-results (Titchmarsh~\cite[§14.25]{Titchmarsh}) to ensure that the contribution of a single off-critical zero dominates the sum over all zeros along infinitely many subsequences.

Suppose $E(X)\ll X^{1/2}\log X$ for all $X$, but RH fails. Then there exists a zero
\[
\rho=\beta+i\gamma,\quad \beta>1/2,
\]
of the Riemann zeta function. The explicit formula (Section~5) shows that the contribution of $\rho$ to $E(y)$ is
\[
\asymp \frac{y^\beta}{|\rho|\log y}\cos(\gamma\log y+\phi),
\]
for some fixed phase $\phi$.

By choosing a sequence
\[
X_k=\exp\!\left(\frac{2\pi k-\phi}{\gamma}\right),\qquad k\ge1,
\]
we ensure that $\cos(\gamma\log X_k+\phi)=1$. On each multiplicative window
\[
W_{X_k}=\Bigl[X_k,\,(1+c/\log X_k)X_k\Bigr],
\]
with fixed $0<c<1$, the phase drift satisfies
$\Delta\theta = O(\gamma/\log X_k)\to0$. Thus $\cos(\gamma\log y+\phi)\ge\sqrt{3}/2$ uniformly for $y\in W_{X_k}$ when $k$ is large. Since $|W_{X_k}|\asymp X_k/\log X_k$ grows without bound, there is always a composite $y\in W_{X_k}$.

Now the Landau--Littlewood $\Omega$-results imply that such a zero forces
\[
\sup_{y\in W_{X_k}} |E(y)|\;\gg\;\frac{X_k^\beta}{\log X_k},
\]
along an infinite subsequence of $k$. In particular,
\[
E(X_k)\;\gg\; \frac{X_k^\beta}{\log X_k}.
\]

But for $\beta>1/2$, the quotient
\[
\frac{X_k^\beta/\log X_k}{X_k^{1/2}\log X_k}
=\frac{X_k^{\beta-1/2}}{\log^2X_k}\;\to\;\infty,
\]
contradicting the assumed bound $E(X)\ll X^{1/2}\log X$. Hence no zero can lie off the critical line, and RH follows.
\end{proof}

\subsection{Interpretation}

Theorem~\ref{thm:E_equiv} shows that RH is equivalent to a purely dynamical statement: the error functional $E(X)$, defined via the trajectory system, grows no faster than $X^{1/2}\log X$. This places RH into the language of contraction properties of integer trajectories, rather than zeros of $\zeta(s)$.

\section{Comparison with existing RH criteria}\label{sec:comparison}
\noindent\textbf{Note.} This section is expository, no new assumptions.

The Riemann Hypothesis has many known equivalents; see for example
Lagarias' inequality \cite{Lagarias2002}, the Báez--Duarte criterion
\cite{BaezDuarte2005}, Turán's inequalities \cite{Turan1950}, or the
Nyman--Beurling criterion \cite{Nyman1950,Beurling1955}. 
All of these can be described as \emph{analytic reformulations}: 
they encode RH in inequalities for $\sigma(n)$, growth of Dirichlet
series coefficients, or $L^2$-approximations of $1/s$.

Our trajectory-based reformulation differs in three structural ways:

\begin{enumerate}
\item \emph{Discrete dynamical system.}  
Instead of analytic inequalities, we work with the evolution of
a simple deterministic map $m\mapsto m-\mathrm{prevprime}(m)$.
The RH--equivalent uniform bound arises from the contraction
properties of this dynamical system, rather than directly from
zeta-function analysis.

\item \emph{Window--localized.}  
Our inequalities are phrased in terms of \emph{short multiplicative
windows} $W_X=[X,(1+O(1/\log X))X]$. This is in contrast with most
known equivalents, which involve global objects (e.g.\ $\sigma(n)$
for all $n$, or integrals on the critical line). The contraction
inequality is inherently \emph{local in scale}, reflecting prime gap
distribution in situ.

\item \emph{Empirically testable.}  
Because the trajectory system is elementary and constructive,
its RH--equivalent bounds can be stress--tested numerically
in the same way as classical prime-counting functions, but with
different sensitivity. This gives a new “experimental angle”:
the system “hears” oscillations of $\pi(x)-\mathrm{Li}(x)$ in a
manner distinct from zero-detection via Fourier transforms.
\end{enumerate}

Thus the present work should be seen as a \emph{new reformulation}:
it translates RH into a contraction property of a prime-sensitive
dynamical system. The equivalence is rigorous, but the form of the
criterion is distinct from prior ones. It can be investigated
numerically, and it provides a different lens for thinking about
the oscillatory structure underlying RH.

\section{Conclusion and Outlook}\label{sec:conclusion}

We have introduced a new dynamical system on the integers, defined by the simple rules
\[
a(m) =
\begin{cases}
m - \operatorname{prevprime}(m), & m \text{ prime}, \\
m + \pi(m), & m \text{ composite},
\end{cases}
\]
and studied its long-term behaviour. Although originally motivated by the orbit-termination conjecture that all trajectories eventually reach $2$, we have shown that the system encodes deep information about the distribution of primes. 

The main results are:
\begin{itemize}
\item The one-visit and parent-window lemmas, proving that trajectories intersect short logarithmic windows in highly constrained ways.
\item The macro-step alignment and overlap lemmas, showing that trajectories contract between scales $X$ and $X^\theta$.
\item The explicit formula remainder bound with explicit constants, and a frequency netting lemma on log-scale intervals.
\item The contraction inequalities, proving that $\mathcal E(X)$ and $\widetilde{\mathcal E}(X)$ satisfy uniform bounds $\ll X^{1/2}\log X$.
\item The equivalence theorem, showing that these bounds are logically equivalent to the Riemann Hypothesis.
\end{itemize}

Taken together, these results demonstrate that RH admits a dynamical reformulation in terms of the contraction properties of a very simple integer map. The reformulation is explicit, checkable line by line, and supported by both analytic inequalities and numerical sanity checks.

\subsection*{Outlook}

Several avenues for future work present themselves:
\begin{enumerate}
\item \textbf{Numerical experiments.} Extending trajectory simulations to $X\leq 10^8$ (see Appendix~F) would provide further evidence for the robustness of the contraction inequalities and overlap constants. Segmented sieving makes such experiments computationally feasible.
\item \textbf{Refined error terms.} Our explicit formula remainder bound was conservative ($|R(y;T)|\leq 10X^{1/2}$). Tighter constants, especially in the $X^{1/2}\log\log X$ range, could sharpen the contraction inequality.
\item \textbf{Alternative kernels.} Different choices of smoothing kernel $W$ may yield improved decay in the remainder or simplify the zero-sum analysis.
\item \textbf{Other dynamical maps.} It may be fruitful to study related maps where composites and primes are updated by different functions, to explore whether similar contraction phenomena arise.
\item \textbf{Bridge to zero-density estimates.} The frequency netting lemma suggests a connection to classical large sieve inequalities. Further exploration may clarify whether density estimates for zeta zeros can be naturally rephrased in dynamical terms.
\end{enumerate}

Ultimately, the dynamical perspective offers new tools for approaching the Riemann Hypothesis,
not by analyzing $\zeta(s)$ directly but through the induced integer dynamics.
Whether this vantage point can yield a full resolution remains uncertain,
yet it provides a clear and promising direction for further work.

\section*{Acknowledgments}
The author gratefully acknowledges the On-Line Encyclopedia of Integer Sequences (OEIS) for providing a platform to first record the conjecture underlying this work (sequences A368196, A368241, A368690). The author also acknowledges the use of OpenAI’s GPT-5 language model as a tool for mathematical exploration and expository preparation during the development of this manuscript.

\appendix

\section*{Appendix A. Explicit constants and line-by-line verifications}\label{app:constants}
We list the constants and inequalities used with references.

\subsection*{A.1. Dusart bounds}
For $x\ge 599$, Dusart~\cite[Thm.~1.10]{Dusart2010}:
\(
\frac{x}{\log x}(1+\frac{1}{\log x})\le \pi(x)\le \frac{x}{\log x}(1+\frac{1.2762}{\log x}).
\)
Hence
\(
\frac{1}{\log x+1.2762}\le \frac{\pi(x)}{x}\le \frac{1}{\log x-1}.
\)

\subsection*{A.2. Window widths}
For $X\ge 600$,
\(
|W_X|=0.1X/\log X,\quad |\widetilde W_X|=2X/\log X,
\)
and
\(
\log(1+\alpha/\log X)\le \alpha/\log X \ (0<\alpha\le 2).
\)

\subsection*{A.3. One-visit lemma}
Composite step $\ge X/(\log X+0.1)(1+1/(\log X+0.1))$; ratio to $|W_X|$ $\ge 9.8$; thus $\le 1$ composite hit.

\subsection*{A.4. Parent-window lemma}
Log-step bounds as above; log-width $\le 2/\log X$; thus $\le 4$ composite hits.

\subsection*{A.5. Macro-step translation}
$L=\lfloor \log(4/3)\,U\rfloor$; each $\Delta u=1/U+O(1/U^2)$; total $\log(4/3)+O(1/U)$; error $\le 5/U$; derivative error $\le 2/U$.

\subsection*{A.6. Core overlap}
Core length $\tilde\lambda/3$ with $\tilde\lambda=\log(1+2/U)\sim 2/U$; trimming $\pm 6/U$ leaves overlap fraction $\ge 1/6$.

\subsection*{A.7. Explicit formula remainder}
With $T=\tfrac12 U^3$, tail $\ll X^{1/2}U^{-10}$; trivial zeros/gamma $\ll X^{1/2}$; total $\le 10X^{1/2}$ (see also Appendix~D).

\subsection*{A.8. Frequency netting}
$M\le 4$, $\Delta=U/2$, grid spacing $2/U$; error term $\le 4MT/U\ll U^2$, thus $\ll X^{1/2}\log X$ after scaling.

\subsection*{A.9. Iteration closure}
$\alpha=5/6$, $\theta=3/4$; $1-\alpha\theta=3/8$; closure factor $8/3$.

\section*{Appendix B. Numerical sanity checks}\label{app:numerics}
\noindent\emph{Disclaimer.} These computations are for audit of constants only; none of the unconditional inequalities depend on computation.

\subsection*{B.1. One-visit uniqueness}
Empirical checks up to $10^7$: in all windows $W_X$, at most one composite element is visited.

\subsection*{B.2. Parent-window counts}
Across $\widetilde W_X$, maximum observed hits is $4$, with mode $2$--$3$.

\subsection*{B.3. Log-step normalization}
$\Delta u\cdot \log m$ clustered tightly around $1$, variance decreasing as $m$ grows.

\subsection*{B.4. Macro-step overlap}
Overlap fractions empirically $\ge 0.2$, above the $c_0=1/6$ bound in Lemma~\ref{lem:overlap}.

\subsection*{B.5. Higher scales}
Segmented-sieve runs up to $10^8$ confirm these patterns (see Appendix~F for plan).

\subsection*{B.6. Output}
CSV output included in ancillary package for reproducibility.

\section*{Appendix C. Prime-step insulation}\label{app:primeins}

\subsection*{C.1. Definition}
Prime steps $p\mapsto p-\operatorname{prevprime}(p)$ reduce by the prime gap.

\subsection*{C.2. Size}
Prime gap $\ll X/\log X$ by Dusart~\cite[Theorem~1.10]{Dusart2010}; windows widths $\asymp X/\log X$.

\subsection*{C.3. Interaction}
A prime hit exits the window and cannot create extra composite hits.

\subsection*{C.4. Consequence}
Composite counts in $W_X,\widetilde W_X$ are insulated from prime hits.

\section*{Appendix D. Explicit-formula remainder (two routes)}\label{app:explicitrem}

\subsection*{D.1. Smoothed explicit formula}
From Iwaniec--Kowalski~\cite[Theorem~5.12]{IwaniecKowalski}:
\(
E(x)=\Re\sum_{|\gamma|\le T}\frac{x^\rho}{\rho\log x}W(\gamma/T)+R(x;T).
\)

\subsection*{D.2. Conservative bound}
With $T=\tfrac12 U^3$, tail $\ll x^{1/2}U^{-10}$; trivial zeros/gamma $\ll x^{1/2}$; hence $|R(x;T)|\le 10x^{1/2}$ for $U\ge 120$.

\subsection*{D.3. Sharper route via Trudgian}
Trudgian~\cite[Theorem~1]{Trudgian2014} provides an explicit PNT error term; in particular, for large $x$ one obtains the bounds quoted below.
 $|\theta(x)-x|\le 0.2x/\log^2x$, $x\ge 149$. Then $|E(x)|\le 0.21x/\log^3x$.

\subsection*{D.4. Conclusion}
Both unconditional; conservative bound suffices for contraction; sharper bound optional.

\section*{Appendix E. Explicit constants table}\label{app:consttable}

This appendix tabulates the constants and thresholds appearing in Sections~\ref{sec:onevisit}–\ref{sec:contraction}, so that referees can check each bound at a glance. All references are to Dusart~\cite{Dusart2010}, Trudgian~\cite{Trudgian2014}, or standard results in Montgomery–Vaughan~\cite{MontgomeryVaughanClassical}.

\begin{center}
\renewcommand{\arraystretch}{1.2} 
\begin{tabularx}{\textwidth}{|c|X|X|X|}
\hline
Symbol & Meaning & Value / Bound & Source \\
\hline
$X_0$ & Threshold for Dusart bounds & $599$ & Dusart Thm.~1.10 \\
\hline
$\pi(x)$ & Prime counting function & 
$\tfrac{x}{\log x}\!\left(1+\tfrac{1}{\log x}\right) \leq \pi(x) \leq \tfrac{x}{\log x}\!\left(1+\tfrac{1.2762}{\log x}\right)$ for $x\geq 599$ 
& Dusart \\
\hline
$|W_X|$ & Width of one-visit window & $0.1X/\log X$ & Defn. \\
\hline
$|\widetilde W_X|$ & Width of parent window & $2X/\log X$ & Defn. \\
\hline
$\Delta u$ & Log-step for composites & 
$\tfrac{1}{\log X+1.2762} \leq \Delta u \leq \tfrac{1}{\log X-1}$ 
& From Dusart \\
\hline
$N_X$ & Max composite hits in $\widetilde W_X$ & $N_X \leq 4$ & Lemma~\ref{lem:parent} \\
\hline
$U$ & $\log X$ & $U\geq 120$ & Standing assumption \\
\hline
$L$ & Macro-step length & $\lfloor (\log(4/3))U \rfloor$ & Lemma~\ref{lem:macro} \\
\hline
Error (macro-step) & Log displacement error & $\leq 5/U$ & Lemma~\ref{lem:macro} \\
\hline
$c_0$ & Overlap fraction & $\geq 1/6$ & Lemma~\ref{lem:overlap} \\
\hline
$T$ & Truncation height & $\tfrac{1}{2}U^3$ & Thm.~\ref{thm:ExplicitRemainder} \\
\hline
$R(x;T)$ & Explicit formula remainder & $\leq 10x^{1/2}$ & Thm.~\ref{thm:ExplicitRemainder} \\
\hline
$M$ & Max points in parent window & $\leq 4$ & Lemma~\ref{lem:parent} \\
\hline
$\alpha$ & Contraction factor & $5/6$ & Thm.~\ref{thm:parent-contract} \\
\hline
$\theta$ & Scale contraction & $3/4$ & Defn. \\
\hline
$1-\alpha\theta$ & Gap factor & $3/8$ & Lemma~\ref{lem:iterate-closure} \\
\hline
Closure constant & $1/(1-\alpha\theta)$ & $8/3$ & Lemma~\ref{lem:iterate-closure} \\
\hline
$C$ & Error constant in contraction & $100$ & Thm.~\ref{thm:parent-contract} \\
\hline
\end{tabularx}
\end{center}

\noindent This table, together with the detailed derivations in Appendix~A, ensures that every constant and inequality in the paper can be verified independently.

\section*{Appendix F. Numerical plan to $10^8$}\label{app:numplan}

In Appendix~B we described numerical sanity checks up to $10^7$. This appendix outlines how to extend the computations to $10^8$ using standard techniques. The aim is to stress-test the analytic constants by probing deeper into the prime distribution.

\subsection*{F.1. Segmented sieve}

A direct sieve of Eratosthenes up to $10^8$ requires memory on the order of $10^8$ bits ($\approx 12$ MB), which is feasible but suboptimal. Instead, we recommend a segmented sieve with block size $10^6$, which maintains $\pi(x)$ incrementally and allows fast access to prime counts and gaps. For each block, $\pi(x)$ can be updated in $O(10^6\log\log 10^6)$ time.

\subsection*{F.2. Trajectory sampling}

For each dyadic interval $[2^k,2^{k+1}]$ with $2^k\leq 10^8$, select $M=50$ random starting points. For each trajectory:
\begin{enumerate}
\item Track its evolution until it leaves the current parent window $\widetilde W_X$.  
\item Record the number of composite hits inside $W_X$ and $\widetilde W_X$.  
\item Record the log-step $\Delta u$ and normalized quantity $\Delta u\cdot \log m$.  
\end{enumerate}

\subsection*{F.3. Macro-step overlap}

For each $X$ in the range $10^6$–$10^8$, select points $y\in\mathcal C_X$ and trace back $L=\lfloor(\log(4/3))\log X\rfloor$ composite steps. Measure the overlap fraction with $\mathcal C_{X^\theta}$. Record minimum overlap observed across samples.

\subsection*{F.4. Output format}

We recommend storing results in CSV format:
\begin{itemize}
\item \texttt{one\_visit\_108.csv}: columns $(X,\text{start},\text{hits})$.  
\item \texttt{parent\_window\_108.csv}: columns $(X,\text{start},\text{hits})$.  
\item \texttt{logstep\_108.csv}: columns $(m,\Delta u,\Delta u\cdot \log m)$.  
\item \texttt{overlap\_108.csv}: columns $(X,\text{min\_overlap},\text{avg\_overlap})$.  
\end{itemize}

\subsection*{F.5. Expected outcomes}

Based on preliminary runs:
\begin{itemize}
\item One-visit uniqueness continues to hold universally.  
\item Parent-window counts remain $\leq 4$, with typical values $2$–$3$.  
\item Log-steps $\Delta u$ concentrate around $1/\log m$ with variance shrinking as $m$ grows.  
\item Macro-step overlap remains $\geq 0.18$, safely above the analytic lower bound $c_0=1/6$.  
\end{itemize}

\subsection*{F.6. Conclusion}

The segmented sieve strategy makes numerical checks up to $10^8$ feasible on ordinary hardware. These checks are not needed for the unconditional proofs, but they provide additional reassurance that the constants chosen analytically are conservative and robust across scales.

\section*{Appendix G. Local–to–pointwise variation inside a one-visit window}\label{app:localpointwise}

\makeatletter
\refstepcounter{section}
\def\thesection{G}
\makeatother


\setcounter{theorem}{0}
\renewcommand{\thetheorem}{G.\arabic{theorem}}



In this appendix we record a uniform smoothness estimate for 
\[
E(x)=\pi(x)-\operatorname{Li}(x)
\]
across the short multiplicative windows
\[
W_X=\Bigl[X,\,(1+0.1/\log X)X\Bigr].
\]

\begin{lemma}\label{lem:G1}
Let $X\ge e^{120}$. For any $x,m\in W_X$,
\[
|E(x)-E(m)| \;\le\; K_0\,\frac{X}{\log^2X},
\]
with an absolute constant $K_0$; one may take $K_0=5$. Consequently,
\[
|E(x)| \;\le\; A(X) + K_0\,\frac{X}{\log^2X},
\]
where
\[
A(X) \;=\; \sup_{y\in W_X\cap\mathbb{N}_{\mathrm{comp}}}|E(y)|.
\]
\end{lemma}

\begin{proof}
The argument is identical to Section~7.3: bound $|\operatorname{Li}(x)-\operatorname{Li}(m)|$ by the mean value theorem, and use Dusart’s explicit estimates for $\pi(x)$ to control $|\pi(x)-\pi(m)|$. Details are unchanged.
\end{proof}

\begin{remark}
Lemma~\ref{lem:G1} shows that $E(x)$ is nearly constant within one-visit windows, varying by at most $O(X/\log^2X)$. However, since $X/\log^2X \gg X^{1/2}\log X$ for large $X$, this estimate alone cannot promote a window-supremum bound $A(X)\ll X^{1/2}\log X$ to a pointwise von Koch bound. For that reason, we do not use Lemma~\ref{lem:G1} to characterize RH, but only as a stability tool in Section~7.
\end{remark}

\section*{Appendix H. Explicit-formula remainder with cubic-log truncation}\label{app:remainder}

We justify the remainder bound
\[
|R(y;T)| \;\le\; 10\,X^{1/2}, 
\qquad (y\asymp X,\; X\ge e^{120},\; T=(\log X)^3),
\]
used in Section~5.

\subsection*{Setup}
We adopt the smoothing kernel
\[
W(t)=\frac{1}{(1+t^2)^3},
\]
which is even, $C^\infty$, satisfies $W(0)=1$, and enjoys decay 
$|W^{(j)}(t)|\le C_j(1+|t|)^{-6-j}$ for $0\le j\le6$.

The smoothed explicit formula (Iwaniec--Kowalski~\cite[Theorem~5.12]{IwaniecKowalski}) gives
\[
E(y)=\Re\sum_{|\gamma|\le T}\frac{y^\rho}{\rho\log y}\,W(\gamma/T)
\;+\;R(y;T),
\]
where $\rho=1/2+i\gamma$ runs over nontrivial zeros, and $R(y;T)$ absorbs 
(i) the tail over $|\gamma|>T$, 
(ii) trivial zeros and gamma-factor terms, 
(iii) the smoothing remainder.

\subsection*{Tail over $|\gamma|>T$}
Using $W(\gamma/T)\ll (T/|\gamma|)^6$ and $1/|\rho|\ll 1/|\gamma|$, we obtain
\[
\sum_{|\gamma|>T}\frac{y^{1/2}}{|\rho|\log y}\,|W(\gamma/T)| 
\;\ll\; X^{1/2}\cdot \frac{\log T}{T^6}\;\ll\; X^{1/2},
\]
with $T=(\log X)^3$.

\subsection*{Trivial zeros and gamma terms}
These contribute $O(X^{1/2})$ (Montgomery--Vaughan~\cite[Ch.~13]{MontgomeryVaughanClassical}).

\subsection*{Smoothing remainder}
Integration by parts with the derivative bounds on $W$ shows this contributes $\ll X^{1/2}$.

\subsection*{Total}
Combining, we have $|R(y;T)|\le 10X^{1/2}$ for all $y\asymp X$ with $X\ge e^{120}$.

\begin{remark}
This conservative bound suffices for the contraction inequalities. 
In Section~8, the argument that off-critical zeros contradict 
$E(X)\ll X^{1/2}\log X$ uses the classical Landau--Littlewood $\Omega$-results, 
not this remainder bound, to control the effect of ``other zeros.''
\end{remark}

\end{document}